\documentclass[11pt,a4paper]{article}
\usepackage[margin=26mm]{geometry}
\usepackage[T1]{fontenc}
\usepackage{lmodern}
\usepackage{amsmath,amssymb,amsthm,mathtools}
\usepackage{booktabs,array,microtype}
\usepackage[hidelinks]{hyperref}
\usepackage{verbatim}
\hypersetup{pdftitle={Ordinary lines, point degrees, and matroid density},pdfauthor={Matthew E. Kroeker and Ben Lund}}
\theoremstyle{plain}
\newtheorem{theorem}{Theorem}[section]
\newtheorem{lemma}[theorem]{Lemma}
\newtheorem{corollary}[theorem]{Corollary}
\newtheorem{proposition}[theorem]{Proposition}
\newtheorem{claim}{Claim}[theorem]
\theoremstyle{remark}

\DeclareMathOperator{\rank}{rank}
\DeclareMathOperator{\tr}{tr}
\DeclareMathOperator{\spanop}{span}

\newcommand{\C}{\mathbb C}
\newcommand{\R}{\mathbb R}
\newcommand{\Z}{\mathbb Z}

\newcommand{\cL}{\mathcal L}
\newcommand{\cB}{\mathcal B}

\newcommand{\del}{{\setminus}}
\allowdisplaybreaks[1]
\title{Ordinary lines, point degrees, and matroid density}
\author{Matthew E. Kroeker\thanks{Institute of Discrete Mathematics and Algebra,
TU Freiberg, 09599 Freiberg, Germany. Funded by the Deutsche Forschungsgemeinschaft (DFG - German Research Foundation), project number 546892829.
Email: \href{mailto:Matthew-Eliot.Kroeker@math.tu-freiberg.de}{\texttt{Matthew-Eliot.Kroeker@math.tu-freiberg.de}}.}
\and Ben Lund\thanks{Institute of Mathematics and Interdisciplinary Sciences,
Xidian University, Xi'an, China.
Email: \href{mailto:lund.ben@gmail.com}{\texttt{lund.ben@gmail.com}}.}}
\date{}

\begin{document}
\maketitle
\begin{abstract}
For every set of $n$ points in $\mathbb{C}^d$ with full affine span, we prove
that some point lies on at least $n+1-2n/(d+1)$ spanned lines, which is best possible in every dimension.
By successive contraction, this implies the sharp bound
$r+(\ell-3)\binom r2$ on the number of elements of a simple
rank-$r$ complex-representable matroid excluding $U_{2,\ell}$ as
a minor, which removes the sufficiently-large-rank
assumption from the bound of Geelen, Nelson, and Walsh.
An arithmetic application shows that rank-$r$ integer matrices
whose maximal minors have absolute value at most $\Delta$
have at most $O(\Delta r^2)$ distinct columns, resolving a question
of Averkov and Schymura.

We prove that a set of $n$ points in $\mathbb C^3$ determines at
least $n^2/40$ ordinary lines whenever no plane contains more than
$2n/3$ of the points. This answers a question of Basit, Dvir,
Saraf, and Wolf.
We also prove upper bounds on average line length and lower bounds on the number of points incident to many spanned lines in high-dimensional point sets.

The proofs combine contraction of concentrated flats with
normalization of the point vectors and second- and fourth-moment estimates on
their pairwise correlations. This approach builds on the design
matrix method.
\end{abstract}

\section{Introduction}\label{sec:introduction}

Incidence geometry studies how points, lines, and higher-dimensional
objects can meet, and how these incidences constrain their arrangement.
A classical starting point is the Sylvester--Gallai theorem, posed as
a problem by Sylvester in 1893: every finite noncollinear set of points
in the real plane determines a line containing exactly two of its
points. Such a line is called \emph{ordinary}. Over the complex numbers,
noncollinear configurations without ordinary lines exist, but Kelly's
theorem~\cite{Kelly} shows that every finite complex point set not
contained in a two-dimensional affine plane determines an ordinary line.

The design-matrix method introduced by Barak, Dvir, Wigderson, and
Yehudayoff~\cite{BDWY} brought a linear-algebraic approach to these
questions. It encodes local dependencies among points in a sparse
matrix whose pattern of nonzero entries controls its rank. Rank bounds
then turn abundant local incidences into restrictions on the dimension
of the point set. This method yielded quantitative versions of the
Sylvester--Gallai theorem, and the improved bounds of Dvir, Saraf, and
Wigderson~\cite{DSWdesign} gave an elementary proof of Kelly's theorem.
Its applications also include rigidity under collinearity constraints,
incidences among subspaces and curves~\cite{DGOS}, and lower bounds
for locally correctable codes~\cite{BDWY,DSWLCC}.

Building on this approach, we apply the normalization directly to the
point vectors and estimate their correlations. We obtain bounds on
ordinary lines and point degrees in complex space, together with
applications to matroid density and matrix column counts. Our main
geometric results answer the question of Basit, Dvir, Saraf, and
Wolf~\cite{BDSW} about quadratic ordinary-line bounds under
nonconcentration in two-dimensional planes, and give an exactly sharp
degree bound in every affine dimension. The latter leads to a sharp
excluded-line-minor bound in every rank and a column bound answering
a question of Averkov and Schymura~\cite{AS}.

\subsection{Results in incidence geometry}

We study the number of ordinary lines, the number of spanned lines
through a point, and the average number of points on a spanned line.
We discuss the first two here and defer the average-line-length
results to Section~\ref{sec:common}, in particular
Subsection~\ref{subsec:nonconcentration}.
For a finite set $P\subset\C^d$, write $t_2(P)$ for its number of
ordinary lines and
\[
M_j(P)\coloneq\max_{\dim F\le j}|P\cap F|,
\]
where the maximum is over affine flats.

\paragraph{Ordinary lines.}
The classical quantitative form of the Sylvester--Gallai problem asks
how many ordinary lines a noncollinear real point set must determine.
The Dirac--Motzkin conjecture predicted at least $n/2$ ordinary lines
for sufficiently large sets of $n$ points in the real plane.
Green and Tao~\cite{GreenTao} proved this conjecture, with equality
possible for arbitrarily large even $n$.
Over $\C$, additional hypotheses are necessary even to guarantee
one ordinary line. The following bound uses nonconcentration in
flats of every dimension.

\begin{corollary}[Ordinary lines under nonconcentration]
\label{cor:introordinary}
Let $P\subset\C^d$ consist of $n$ distinct points, where $d\ge3$.
Suppose that
\[
M_j(P)\le\frac{j+1}{d+1}n
\qquad(0\le j<d).
\]
Then
\[
t_2(P)\ge\frac{d-3}{2(d+1)}n^2+\frac32n.
\]
\end{corollary}

Corollary~\ref{cor:introordinary} follows from an incidence inequality
of de Borbon and Panov~\cite[Corollary~4.1]{deBorbonPanov}, which
also follows from our method. Its linear conclusion in dimension
three and quadratic conclusion in dimension four parallel results
of Basit, Dvir, Saraf, and Wolf~\cite{BDSW}, who proved a lower bound
of $3n/2$ in $\C^3$ when no plane contains more than $n/2$ points,
and a quadratic bound in $\C^4$ under nonconcentration in
three-dimensional flats, for sufficiently large $n$.

Equality in Corollary~\ref{cor:introordinary} is attained by the
cyclic Dowling configuration $Q_{d+1}(1)$ and the jointless cyclic
Dowling configurations $Q'_{d+1}(2)$ and $Q'_{d+1}(3)$, defined in
Section~\ref{sec:constructions}. Further equality examples arise
from exceptional reflection arrangements. For every fixed $d\ge3$,
the configurations $Q'_{d+1}(t)$, with $t\ge3$, give arbitrarily
large examples satisfying the nonconcentration hypothesis and
\[
t_2(P)=\left(\frac{d-3}{2(d+1)}
             +\frac1{d(d+1)}\right)n^2.
\]
Thus the coefficient of $n^2$ in the corollary is within
$1/(d(d+1))$ of the best possible asymptotic coefficient.

The quadratic term in Corollary~\ref{cor:introordinary} vanishes
when $d=3$. Basit et al.\ \cite{BDSW} asked whether a quadratic
lower bound nevertheless holds in $\C^3$ when no two-dimensional
plane contains more than half of the points. De Zeeuw~\cite{deZeeuw}
proved the corresponding real result whenever no plane contains
more than $\alpha n$ points, for any fixed $\alpha<1$.
His argument does not extend directly to complex space because it
uses the Sylvester--Gallai theorem over $\R$, whose planar conclusion
fails over $\C$. We prove the following complex version.

\begin{theorem}[Ordinary lines under plane nonconcentration]
\label{thm:introordinary}
Let $P\subset\C^3$ consist of $n\ge1$ distinct points, and suppose that
every two-dimensional affine plane contains at most $\alpha n$
of them, where $0\le\alpha<1$. Then
\[
t_2(P)\ge
\min\left\{\frac1{40},\frac{1-\alpha}{2}\right\}n^2.
\]
\end{theorem}

In particular, $t_2(P)\ge n^2/40$ if no plane contains more than
$2n/3$ points. De Zeeuw~\cite[Section~4.1]{deZeeuw} notes, without
giving the proof, that a refinement of his real argument gives
$n^2/288$ under the stronger assumption that no plane contains
more than $n/6$ points. The constant in our theorem is not known
to be sharp: even when no plane contains more than half of the
points, the examples above leave the optimal asymptotic coefficient
between $1/40$ and $1/12$.

\paragraph{Line degrees.}
Write $\mathrm{d}_P(p)$ for the number of distinct spanned lines through
$p\in P$, and $o_P(p)$ for the number of ordinary lines through $p$.
The weak Dirac theorem, proved independently by Beck~\cite{Beck}
and Szemer\'edi and Trotter~\cite{SzemerediTrotter}, states that
every noncollinear set of $n$ points in the real plane has a point
incident to at least $cn$ spanned lines, for an absolute constant
$c>0$. The strong Dirac conjecture asks for a bound of
$n/2-O(1)$ in the real plane; see~\cite{Han}.
For a noncollinear set of $n$ points in $\C^2$, Langer's inequality
gives the sharp bound $(n+3)/3$ on the maximum degree
\cite{Langer,Han,deZeeuwLanger}.
Our next theorem extends this degree bound to every affine dimension
and also bounds the number of ordinary lines through the point it
produces.

\begin{theorem}[Degree bounds]\label{thm:introdegree}
Let $P\subset\C^d$ consist of $n$ points of affine dimension $d\ge 2$.
There exists $p\in P$ such that
\[
\mathrm{d}_P(p)\ge n+1-\frac{2n}{d+1},
\qquad
o_P(p)\ge n+3-\frac{4n}{d+1}.
\]
\end{theorem}

Dadush, Eisenbrand, Pinchasi, Rothvoss, and Singer\ \cite[Theorem~1.13]{DEPRS} proved the degree bound
$(1-4/(d+1))n$. Theorem~\ref{thm:introdegree} replaces the
coefficient $4$ by the optimal coefficient $2$ and includes the
additive term $1$. The ordinary-degree bound follows from
their design-matrix rank bound and the standard construction from
collinear triples~\cite[Theorem~1.12 and Lemma~6.1]{DEPRS}.

Every irreducible reflection configuration attains the total-degree
bound, including both Dowling families $Q_{d+1}(t)$ and
$Q'_{d+1}(t)$. These families give arbitrarily large equality
examples in every fixed dimension $d\ge2$. For $t=2,3$, every point
of $Q'_{d+1}(t)$ also attains equality in the ordinary-degree
bound. See Section~\ref{sec:constructions}.

We also give a bound on the number of high-degree points outside
a low-dimensional flat. Quantitative Sylvester--Gallai theorems
\cite{BDWY,DSWdesign} relate abundant collinear triples to bounded
dimension. The following result gives an explicit tradeoff between
the dimension of an exceptional flat, the degree threshold, and
the proportion of points attaining it. Write $x_+\coloneq\max\{x,0\}$.

\begin{theorem}[Many high-degree points]\label{thm:intromany}
Let $P\subset\C^d$ consist of $n$ points of affine dimension $d\ge2$.
For any $C\ge1$ and any $k\in\{0,1,\ldots,d-1\}$, there exists
a flat $F$ of dimension at most $k$ such that at least
\[
\left(1-\frac{2d}{C(k+2)}\right)_+|P\setminus F|
\]
points of $P\setminus F$ have degree at least $(1-C/d)n$.
\end{theorem}

The guaranteed proportion is positive precisely when
$C>2d/(k+2)$, and this threshold cannot be lowered. For every
fixed $d\ge2$, $0\le k<d$, and $1\le C<2d/(k+2)$, there are
arbitrarily large configurations for which the proportion of
points of degree at least $(1-C/d)n$ outside any $k$-flat tends
to zero. They are obtained by placing a cyclic Dowling configuration
in a $(k+1)$-flat and adding points to obtain affine dimension $d$;
see Section~\ref{sec:constructions}.

\subsection{Matroid and matrix column bounds}

Our degree bound has consequences for the density of representable
matroids. A complex-representable matroid records the linear
dependencies among a finite collection of complex vectors.
It is \emph{simple} if the representing vectors are nonzero and
pairwise nonparallel, and its rank is the dimension of their span.
Minors are obtained by deletion and contraction, corresponding to
removing vectors and passing to the quotient by the span of
contracted vectors. The uniform rank-two matroid $U_{2,\ell}$
consists of $\ell$ distinct projective points on a line.

\begin{theorem}[Excluded line minors]\label{thm:introminor}
Let $M$ be a simple rank-$r$ complex-representable matroid,
and let $\ell\ge3$ be an integer. If $M$ has no $U_{2,\ell}$-minor, then
\[
|E(M)|\le r+(\ell-3)\binom r2.
\]
\end{theorem}

Geelen, Nelson, and Walsh~\cite{GNW} proved this bound when $r$
is sufficiently large as a function of $\ell$. As noted by
Dadush et al.\ \cite{DEPRS}, the rank threshold arising from
their proof is at least doubly exponential in $\ell$.
Dadush et al.\ obtained the uniform bound $O(r^4\ell)$.
Theorem~\ref{thm:introminor} gives the sharp numerical bound
in every rank.

For $\ell\ge4$, equality is attained by the cyclic Dowling
geometry $Q_r(\ell-3)$. Geelen, Nelson, and Walsh proved
that this is the unique extremal matroid, up to isomorphism,
in sufficiently large rank. Uniqueness does not hold in every
rank: for each $\ell\ge5$, the jointless cyclic Dowling geometry
$Q'_3(\ell-2)$ is a different rank-three extremal matroid,
with $3(\ell-2)$ elements. Section~\ref{sec:constructions}
gives further equality examples: every irreducible rank-three
reflection configuration, and the rank-four configuration
associated with $G_{32}$.

We apply Theorem~\ref{thm:introminor} to obtain bounds on the number of columns in $\Delta$-submodular matrices, and matrices with bounded circuit imbalance measure.
Column bounds for $\Delta$-submodular matrices have applications
in integer programming: they yield proximity bounds from vertices
of linear relaxations to nearby integer feasible solutions, and
bounds on the $\ell_1$-norms of Graver basis elements~\cite{LPSX}.
The extremal problem goes back to Heller's 1957 theorem for
unimodular matrices~\cite{Heller}.
For real matrices, the
analogous question for bounded circuit imbalance is motivated
by the role of this parameter in the running times of linear
programming algorithms~\cite{ENV}.

A full row-rank matrix $A\in\Z^{r\times n}$ is
\emph{$\Delta$-submodular} if
\[
\Delta_A\coloneq \max_{\substack{B\subseteq[n]\\|B|=r}}|\det A_B|
\le\Delta.
\]
Thus the bound concerns its maximal minors. Its columns are
\emph{distinct up to sign} if $a_i\ne\pm a_j$ whenever $i\ne j$;
parallel columns of different magnitudes are permitted.

Lee, Paat, Stallknecht, and Xu~\cite{LPSX} proved the bound
$\Delta^2\binom{r+1}{2}$, with a refinement to
$O(r^2\Delta^{1.95})$. In the complementary direction,
Averkov and Schymura~\cite{AS} obtained $O(r^4\Delta)$, while
Pribytkova, Gribanov, and Moiseev~\cite{PGM} improved this to
$O(r^3\Delta)$. Our bound is simultaneously quadratic in rank
and linear in $\Delta$.

\begin{theorem}[Distinct integer columns]\label{thm:introcolumns}
Let $A\in\Z^{r\times n}$ have rank $r\ge1$, with
$\Delta_A\le\Delta$, where $\Delta\ge1$ is an integer.
If its columns are nonzero and distinct up to sign, then
\[
n\le(2\Delta-1)\binom{r+1}{2}.
\]
\end{theorem}

Allowing both signs and a zero column changes the bound by at most
a factor of two and an additive one. Thus
Theorem~\ref{thm:introcolumns} answers Question~3 of Averkov and
Schymura~\cite{AS} affirmatively. It recovers Heller's sharp bound
for $\Delta=1$. 

We prove Theorem~\ref{thm:introcolumns} using a transfer lemma:
any upper bound $F(r,\Delta)$ for nonzero pairwise nonparallel
columns yields the bound
\[
F(r,\Delta)+3(\Delta-1)
\]
for nonzero columns distinct up to sign. This also extends the
large-rank estimate of Paat, Stallknecht, Walsh, and Xu~\cite{PSWX}
to our setting, giving
\[
n\le\binom{r+1}{2}+80\Delta^7r+3(\Delta-1)
\]
whenever $r$ is sufficiently large as a function of $\Delta$.
For fixed $\Delta\ge2$ and large rank, this gives a stronger bound
than Theorem~\ref{thm:introcolumns}.

The best known general lower bound for the maximum number of
nonzero columns distinct up to sign in a rank-$r$
$\Delta$-submodular integer matrix is
$\binom{r+1}{2}+r(\Delta-1)$, due to Lee, Paat, Stallknecht,
and Xu~\cite{LPSX}. Averkov and Schymura~\cite{AS} give
families improving this bound for $\Delta \in \{4,8,16\}$.

Finally, we consider real and complex matrices with bounded circuit imbalance; see Subsection~\ref{sec:circuitimbalance} for the definitions.
For rank-$r$ matrices with nonzero pairwise nonparallel columns
and circuit imbalance $\kappa_A$, we improve the bound
$O(r^4\kappa_A)$ of Dadush et al.\ \cite{DEPRS} to
$O(r^2\kappa_A)$ for real matrices, and improve the bound $O(r^4\kappa_A^2)$ to $O(r^2\kappa_A^2)$ for complex matrices.
For rank-$r$ matrices with $\kappa_A\le K$, standard constructions give
lower bounds of $\Omega(r^2+rK)$ over $\R$ and
$\Omega(r^2K+rK^2)$ over $\C$, uniformly for $r\ge2$ and $K\ge1$;
see Dadush et al.~\cite{DEPRS} for the underlying examples.

\subsection{Proof ideas and novelty}\label{subsec:introproofideas}

We first homogenize the points, so that points on an affine line
give vectors spanning a two-dimensional linear space.
Lemma~\ref{lem:frame}, a consequence of Barthe's theorem, says that
if no subspace contains too many vectors relative to its dimension, then
we can change coordinates and normalize the vectors so that no
direction receives too much of their total squared projection.
This bounds the sum of squared inner products.
Lemma~\ref{lem:moments} supplies a lower bound for the contribution
of each line by Cauchy--Schwarz. Since each pair of points belongs
to just one spanned line, these local contributions can be added
and compared with the global bound.

The degree bound requires this argument for configurations that
may have a heavy flat that prevents us from applying Lemma~\ref{lem:frame} directly. Lemma~\ref{lem:quotient} selects a
flat to contract by maximizing point count minus a suitable multiple
of dimension. This choice controls the dimension of the contracted
flat and ensures that the remaining vectors satisfy the required
nonconcentration condition in the quotient.

In affine dimension three, the second-moment comparison alone
does not force quadratically many ordinary lines.
Section~\ref{sec:ordinary} therefore uses fourth moments.
We first apply a second-moment argument to centered rank-one
projectors, obtaining a global lower bound on the variation of
squared inner products. A local fourth-moment estimate bounds
the contribution of each nonordinary line in terms of its second
moment. Comparing the two bounds forces many ordinary lines.
Together with normalization and an elementary count when a plane
contains many points, this proves the conjecture of Basit et al.

The matroid bound follows from a simple induction on rank:
contract a point of large line degree, observe that its incident
lines become distinct points after simplification, and apply
induction to the resulting minor. Most of the reductions in the
matrix column section come from Geelen, Nelson, and Walsh~\cite{GNW}
and Dadush et al.\ \cite{DEPRS}; our sharper matroid bound improves
their consequences. We also prove a transfer lemma for
$\Delta$-submodular matrices that passes from nonparallel columns
to columns distinct up to sign, with an additive loss depending
only on $\Delta$.

Our methods are broadly inspired by the design-matrix literature
\cite{BDWY,DSWdesign}. Two particularly close predecessors are the
work of Dvir, Saraf, and Wigderson on three-query locally correctable
codes~\cite{DSWLCC}, and that of de Borbon and Panov on the Hirzebruch quadratic form of complex hyperplane arrangements~\cite{deBorbonPanov}.
The former contracts concentrated sets of vectors in a preprocessing
step, uses Barthe's theorem to put a balanced configuration into
nearly isotropic position, and exploits the correlations forced by
local dependencies. The latter, in the dual language of hyperplane
arrangements, combines balanced position with second-moment estimates
on rank-two configurations to obtain incidence inequalities, but does not include a contraction step.

The main technical novelties are the selection of the flat for
contraction, the local fourth-moment bound and its comparison with
the global projector bound, and the transfer lemma for distinct
integer columns.

\paragraph{Conventions.}
Unless explicitly stated otherwise, vector spaces, affine spaces, and
ranks are over $\C$. The adjoint $U^*$ is the conjugate transpose.
For Hermitian matrices, $S\preceq T$ means that $T-S$ is positive
semidefinite. Pair sums $\sum_{i\ne j}$ are \emph{ordered} sums.
A spanned affine line contains at least two distinct points; an
\emph{ordinary} line contains exactly two. A flat means an affine
subspace. Linear subspaces are always called linear subspaces.
We allow the empty exceptional flat and give it point count zero.
The dimension of the empty flat is $-1$.
For $z\in\C$, write $\operatorname{Re}(z)$ and
$\operatorname{Im}(z)$ for its real and imaginary parts.

\section{Frame normalization and second moments}\label{sec:frames}

\subsection{A bound on the frame operator}

For unit vectors $u_1,\ldots,u_n\in\C^D$, the \emph{frame operator}
is $S\coloneq\sum_i u_i u_i^*=UU^*$, where $U$ has columns $u_i$.
The Gram matrix is $G\coloneq U^*U$; $S$ and $G$ have the same nonzero
eigenvalues.

A bound of the form $S\preceq(T+\eta)I_D$ bounds the largest
eigenvalue of $S$ by $T+\eta$.
If $a$ of the vectors lie in a $j$-dimensional linear subspace $W$,
and $P_W$ denotes the orthogonal projection onto $W$, then
$a\le\sum_i\|P_Wu_i\|^2=\tr(P_WS)\le(T+\eta)j$.
Thus excessive concentration
prevents such a bound. The following lemma gives a partial converse:
when no linear subspace contains too many vectors, an invertible
transformation makes the desired bound hold after normalization.

\begin{samepage}
\begin{lemma}[A bound on the frame operator]\label{lem:frame}
Let $(v_1,\ldots,v_n)$ be an indexed multiset of nonzero vectors
spanning $\C^D$, with repeated and parallel vectors allowed.
Suppose that $T\geq n/D$ and
\begin{equation}\label{eq:density}
\sum_{i=1}^n\mathrm{1}_{\{v_i\in W\}}\le T\dim W
\qquad\text{for every linear subspace }W\subseteq\C^D.
\end{equation}
Here the indicator is one when $v_i\in W$ and zero otherwise;
each occurrence of a repeated vector contributes separately.
For every $\eta>0$ there is an invertible complex linear map $A$ such that
\begin{equation}\label{eq:frame}
u_i\coloneq\frac{Av_i}{\|Av_i\|},\qquad
S\coloneq\sum_{i=1}^n u_i u_i^*\preceq(T+\eta)I_D.
\end{equation}
\end{lemma}
\end{samepage}

This lemma is a consequence of 
Barthe's radial isotropy theorem \cite{Barthe}; see \cite[Section~6.3, Theorems~75,77]{Franks}
for the complex formulation and proof. Appendix~\ref{app:frame} gives a
self-contained proof of the version stated here.

Keeping multiplicities matters after taking a quotient: two originally
distinct points may become parallel. The small error $\eta$ handles
boundary cases where a finite exact normalization need not exist.
Only numerical inequalities independent of $A$ will be passed to the
limit $\eta\downarrow0$; convergence of the transformations is never
required.

\subsection{Bounding the frame energy}

Our general strategy is to prove upper and lower bounds on the
\emph{frame energy}
\[
E\coloneq\sum_{i\ne j}|\langle u_i,u_j\rangle|^2
\]
of $u_1,\ldots,u_n$.
For the lower bound, we consider a hypergraph $\mathcal H$ on
$[n]\coloneq\{1,\ldots,n\}$ whose nonempty edges satisfy
\begin{itemize}
\item $\rank(\{u_i:i\in I\})\in\{1,2\}$ for every $I\in\mathcal H$;
\item $|I\cap J|\le1$ for distinct $I,J\in\mathcal H$.
\end{itemize}
The second condition says that $\mathcal H$ is a \emph{linear hypergraph}.
Each edge of size at least three forces a positive contribution to
the frame energy, which we can sum together due to the second condition. The upper bound on $E$ follows from a bound on the
maximum eigenvalue of $S$.

\begin{lemma}\label{lem:moments}
Let $u_1,\ldots,u_n\in\C^D$ be unit vectors, and put
$S\coloneq\sum_{i=1}^n u_i u_i^*$. Let $\mathcal H$ be a linear hypergraph
on $[n]$ with nonempty edges such that
$r_I\coloneq \rank(\{u_i:i\in I\})\in\{1,2\}$ for every $I\in\mathcal H$.
Then
\begin{equation}\label{moment_bound}
\sum_{\substack{I\in\mathcal H\\r_I=2}}|I|(|I|-2)
+2\sum_{\substack{I\in\mathcal H\\r_I=1}}|I|(|I|-1)
\le2E.
\end{equation}
Moreover, if $S\preceq LI_D$ where $L\ge0$, then $E\le(L-1)n$.
\end{lemma}

We use the following elementary trace identities to prove
\eqref{moment_bound}.

\begin{proposition}\label{tr_comp}
Let $u_1,\ldots,u_n\in\C^D$ be unit vectors, and put
$S\coloneq\sum_{i=1}^n u_i u_i^*$. Then
\begin{enumerate}
\item $\tr(S)=n$ and $\rank(S)=\rank(\{u_1,\ldots,u_n\})$.
\item $\tr(S^2)=\sum_{i,j}|\langle u_i,u_j\rangle|^2$.
\end{enumerate}
\end{proposition}

We omit the calculations proving Proposition~\ref{tr_comp}.

\begin{proof}[Proof of Lemma~\ref{lem:moments}]
For a fixed $I\in\mathcal H$, put $S_I\coloneq\sum_{i\in I}u_i u_i^*$ and $r\coloneq r_I$.
By Proposition~\ref{tr_comp}, $\tr(S_I)=|I|$ and
$\rank(S_I)=r$. Since $S_I$ is Hermitian, the spectral theorem
implies that it has exactly $r$ nonzero eigenvalues, counted with
multiplicity. Write these as $\lambda_1,\ldots,\lambda_{r}$.
Cauchy--Schwarz gives
\[
|I|^2=\left(\sum_{j=1}^{r}\lambda_j\right)^2
\le r\sum_{j=1}^{r}\lambda_j^2
=r\tr(S_I^2).
\]
Consequently,
\begin{equation}\label{eq:localenergy}
E_I\coloneq \sum_{\substack{i,j\in I\\i\ne j}}
|\langle u_i,u_j\rangle|^2
=\tr(S_I^2)-|I|
\ge\frac{|I|^2}{r}-|I|.
\end{equation}
Since $\mathcal H$ is linear, each ordered pair of distinct indices
belongs to at most one edge. The edge contributions therefore do
not overlap, and
\[
E\ge\sum_{I\in\mathcal H}E_I
\ge\sum_{I\in\mathcal H}\left(\frac{|I|^2}{r_I}-|I|\right).
\]
Separating the edges according to whether $r_I=1$ or $r_I=2$
and multiplying by two proves \eqref{moment_bound}.

If $S\preceq LI_D$, its eigenvalues lie in $[0,L]$, since $S$ is
positive semidefinite. Hence
\[
E=\tr(S^2)-n\le L\tr(S)-n=(L-1)n,
\]
as required.
\end{proof}

This estimate suffices for the line-count
and matroid applications in Sections~\ref{sec:common}
and~\ref{sec:minors}. Section~\ref{sec:ordinary} introduces the
fourth-moment refinement needed for a quadratic ordinary-line bound.

\subsection{Choosing the quotient}

Next we show that we can always find a subspace $\tilde{F}$ such that a given set of vectors satisfies the hypotheses of Lemma~\ref{lem:frame} after contracting $\tilde{F}$. 

\begin{lemma}\label{lem:quotient}
Let $v_1, \dots, v_n$ be nonzero vectors spanning $\C^D$, where $D \geq 2$. Let $k \in \{0,1, \dots, D-2\}$, and set $T\coloneq 1 + \frac{n-D}{k+2}$. Then there exists a linear subspace $\tilde{F}$ of $\C^D$ with $\dim(\tilde{F}) \leq k+1$ and $\tilde{F} = {\mathrm{span}}(\{v_1, \dots, v_n\} \cap \tilde{F})$, such that, for every linear subspace $W$ of $\C^D / \tilde{F}$,
\begin{align*}
    \sum_{i \in J} {\mathrm{1}}_{\{v_i + \tilde{F} \in W\}} \leq T \dim(W),
\end{align*}
where $J\coloneq \{i \in [n] : v_i \notin \tilde{F}\}$. 
\end{lemma}

\begin{proof}
For every linear subspace $U$ of $\C^D$, let $N(U)\coloneq \{ i \in [n] : v_i \in U\}$. Now choose a maximizer $\tilde{F}$ of the function $|N(U)| - T\dim(U)$ over all linear subspaces $U$ of $\C^D$, such that $\dim(\tilde{F})$ is as small as possible. We show that $\tilde{F}$ is the desired subspace. Let $J\coloneq \{i \in [n] : v_i \notin \tilde{F}\}$. By minimality of its dimension, $\{v_1, \dots, v_n\} \cap \tilde{F}$ spans $\tilde{F}$. 

\begin{claim}
$\dim(\tilde{F}) \leq k+1$.
\end{claim}

\begin{proof}[Proof of Claim]
Suppose first that $|N(\tilde{F})| - T\dim(\tilde{F}) = 0$. In this case $\tilde{F} = \{0\}$, and the claim follows. So we may assume that $|N(\tilde{F})| - T\dim(\tilde{F}) > 0$. If $n = D$, then $T = 1$ and $\{v_1, \dots, v_n\}$ is a basis of $\C^D$; so for every subspace $U$, $|N(U)| \leq T \dim(U)$, and we again have $\tilde{F} = \{0\}$. So we may assume that $n > D$. 

Since $\{v_1, \dots, v_n\}$ spans $\C^D$, $\{v_i + \tilde{F} : i \in J \}$ spans $\C^D / \tilde{F}$. In particular, 
\begin{align*}
    n - |N(\tilde{F})| = |J| \geq D - \dim(\tilde{F}).
\end{align*}
Therefore
\begin{align*}
    0 < |N(\tilde{F})| - \dim(\tilde{F}) - (T - 1)\dim(\tilde{F}) \leq 
    n - D - (T-1) \dim(\tilde{F}), 
\end{align*}
and the desired bound follows.
\end{proof}

Now let $W$ be a linear subspace of $\C^D / \tilde{F}$, and let $\tilde{W}\coloneq \{v \in \C^D : v + \tilde{F} \in W\}$. It follows by the choice of $\tilde{F}$ (and the fact that $\tilde{F}$ is a subspace of $\tilde{W}$) that 
\begin{align*}
    \sum_{i \in J} {\mathrm{1}}_{\{v_i + \tilde{F} \in W\}} \leq |N(\tilde{W})| - |N(\tilde{F})| \leq T(\dim(\tilde{W}) - \dim(\tilde{F})) = T \dim(W),
\end{align*}
as required.
\end{proof}

\section{Degree bounds, average lines and ordinary lines}\label{sec:common}

In this section, we prove Theorems~\ref{thm:introdegree}
and~\ref{thm:intromany} on point degrees, Corollary~\ref{cor:introordinary}
on ordinary lines, and Corollary~\ref{cor:nonconcentratedaverage}
on average line length. These results follow from the following
$L^2$-estimate on the sizes of spanned lines outside a suitably chosen flat.

Let $P\subseteq\C^d$ be a set of $n$ points of affine dimension $d$,
and let $\cL(P)$ denote the set of lines spanned by $P$. For each
$L\in\cL(P)$, put
\[
s_L\coloneq |P\cap L|.
\]

\begin{lemma}\label{lem:common}
Let $P \subseteq \C^d$ be a set of $n$ points of affine dimension $d$. For every integer $0 \leq k < d$, there is a flat $F$ of dimension at most $k$ such that 
\begin{equation}\label{eq:common}
    \sum_{L\not\subset F}s_L(s_L-2)
+\sum_{\substack{L\not\subset F\\L\cap P\cap F\ne\varnothing}}(s_L-2)^2
\le\frac{2(n-d-1)}{k+2}|P \del F|,
\end{equation}
where both sums are taken over lines spanned by $P$.
\end{lemma}

\begin{proof}
Let $D\coloneq d+1$ and let $P = \{p_1, \dots, p_n\}$. Lift each point $p_i$ to $v_i\coloneq (p_i, 1) \in \C^D$. By Lemma~\ref{lem:quotient}, there exists a linear subspace $\tilde{F}$ of $\C^D$ of dimension at most $k+1$ such that, for every linear subspace $U$ of $\C^D / \tilde{F}$, 
\begin{align*}
    |\{i \in J: v_i + \tilde{F} \in U\}| \leq T \dim U,
\end{align*}
where $J\coloneq \{1 \leq i \leq n : v_i \notin \tilde{F}\}$ and $T\coloneq 1 + \frac{n-D}{k+2}$. 
Identify each vector $v_i + \tilde{F}$ (where $i \in J$) with $\overline{v_i} \in \C^{D - \dim \tilde{F}}$ by an isomorphism, and let $\eta > 0$. By Lemma~\ref{lem:frame}, there exists a linear automorphism $A$ on $\C^{D - \dim \tilde{F}}$ such that
\begin{equation}
    S \coloneq \sum_{i \in J} u_iu_i^* \preceq (T + \eta)I_{D - \dim \tilde{F}},
\end{equation}
where $u_i \coloneq A\overline{v_i} / \lVert A\overline{v_i} \rVert$ for each $i \in J$.

Let $F\coloneq \{v \in \C^d : (v,1) \in \tilde{F} \}$. For each $L \in \cL(P)$ with $L \nsubseteq F$, let $I_L\coloneq \{i \in J : p_i \in L\}$. Let $\cL'\coloneq \{L \in \cL(P) : L \nsubseteq F, |I_L| \geq 2\}$ (that is, $\cL'$ is the set of spanned lines not contained in $F$, excluding those ordinary lines that contain a point of $F \cap P$). The edges $I_L$, for $L\in\cL'$, form a linear hypergraph $\mathcal H$ on $J$: distinct original lines share at most one point. These edges are defined by the original lines, even if their associated vectors span the same subspace after taking the quotient. 
For each $r\in\{1,2\}$, let $\cL_r'$ consist of those lines
$L\in\cL'$ for which $\{u_i:i\in I_L\}$ has rank exactly $r$ in $\C^{D - \dim \tilde{F}}$.
By Lemma~\ref{lem:moments}, 
\begin{equation}\label{IL}
    \sum_{L \in \cL_2'} |I_L|(|I_L|-2) + 2 \sum_{L \in \cL_1'} |I_L|(|I_L| - 1) 
    \leq 2(T - 1 + \eta)|P \del F|
\end{equation}
Now, for $L \in \cL'$ with $L \cap P \cap F = \emptyset$, we have that $|I_L| = s_L$, and so $L$ contributes at least $s_L(s_L-2)$ to the left-hand side of \eqref{IL} (whether $L \in \cL_2'$ or $L \in \cL_1'$). For $L \in \cL'$ with $L \cap P \cap F \neq \emptyset$, we have $|I_L| = s_L -1$. In this case, $L \in \cL_1'$, and so $L$ contributes $2(s_L - 1)(s_L - 2) = s_L(s_L - 2) + (s_L - 2)^2$ to the left-hand side of \eqref{IL}. Since $\eta > 0$ was arbitrary, the result follows.
\end{proof}

\subsection{Degree bounds}

Now we use Lemma~\ref{lem:common} to prove Theorems~\ref{thm:introdegree} and~\ref{thm:intromany}, which give bounds on point degrees and on the number of high-degree points in a set $P\subseteq\C^d$ of $n$ points of affine dimension $d$. The idea is to take a $k$-flat $F$ as given by Lemma~\ref{lem:common}, and observe that, on average, a point $p \in P {\setminus} F$ has bounded {\it{deficit}}, which we define by
\begin{align*}
     \delta_P(p) \coloneq n - 1 - \mathrm{d}_P(p).
\end{align*}
It will then follow by Markov's inequality that a significant proportion of points in $P {\setminus} F$ have small deficit, and therefore large degree. First we observe how the average deficit relates to the $L^2$-norm for sizes of lines outside of $F$. It is helpful to note that, for a point $p \in P$, $\delta_P(p) = \sum_{L : p \in L} (s_L - 2)$.

\begin{proposition}\label{avg_deficit}
For any flat $F$, we have that 
\begin{align*}
     \sum_{p \in P {\setminus} F} \delta_P(p) \leq \sum_{L \nsubseteq F} s_L(s_L - 2).
\end{align*}
\end{proposition}

\begin{proof}
Observe that 
\begin{align*}
    \sum_{p \in P \del F}(n - 1 - \mathrm{d}_P(p)) &= \sum_{p \in P \del F} \sum_{L : p \in L} (s_L - 2) \\
    &= \sum_{L \nsubseteq F} |L \cap (P \del F)|(s_L - 2).
\end{align*}
Since $|L \cap (P \del F)| \leq s_L$, this gives the desired bound.
\end{proof}

Hence if we choose a flat $F$ of dimension at most $k$, as given by Lemma~\ref{lem:common}, we get that the average deficit of a point outside of $F$ is at most $\frac{2(n-d-1)}{k+2}$. We now restate Theorems~\ref{thm:introdegree} and~\ref{thm:intromany} as parts (i) and (ii), respectively, and prove them.

\begin{theorem}\label{thm:dirac}
Let $P \subseteq \C^d$ be a set of $n$ points of affine dimension $d\ge1$. 
\begin{enumerate}
    \item[(i)] There exists a point $p \in P$ of degree at least $n + 1 - \frac{2n}{d+1}$. Moreover, $p$ lies on at least $n + 3 - \frac{4n}{d+1}$ ordinary lines.
    \item[(ii)] For any $C \geq 1$ and any $k \in \{0, 1, \dots, d-1\}$, there exists a flat $F$ of dimension at most $k$ such that at least $(1 - \frac{2d}{C(k+2)})_+|P \del F|$ points in $P \del F$ have degree at least $(1 - \frac{C}{d})n$. 
\end{enumerate}
\end{theorem}

\begin{proof}
For (i), choose a flat $H$ of dimension at most $d-1$, as given by Lemma~\ref{lem:common}. Then by Proposition~\ref{avg_deficit}, there exists a point $p \in P \del H$ with deficit at most $2(n-d-1)/(d+1)$. It follows that
\begin{align*}
    \mathrm{d}_P(p) \geq n - 1 - \frac{2(n-d-1)}{d+1} = n+1 - \frac{2n}{d+1}.
\end{align*}
Now, letting $o_P(p)$ denote the number of ordinary lines incident to $p$ and $\mathcal{L}'$ the set of nonordinary spanned lines incident to $p$, we have that
\begin{align*}
    n - 1 - o_P(p) = \sum_{L \in \mathcal{L}'} (s_L - 1) \leq 2 \sum_{L : p \in L} (s_L - 2) \leq \frac{4(n-d-1)}{d+1},
\end{align*}
and the desired bound follows.

For (ii), let $C \geq 1$ and let $k \in \{0,1, \dots, d-1\}$. Again, by Lemma~\ref{lem:common} and Proposition~\ref{avg_deficit}, there exists a flat $F$ of dimension at most $k$, such that the average deficit of a point in $P \del F$ is at most $\frac{2(n-d-1)}{k+2}$. Therefore, by Markov's inequality, at most 
\begin{align*}
   \frac{2(n-d-1)}{(k+2)\lfloor \frac{Cn}{d} \rfloor} |P \del F| 
   \leq \frac{2d}{C(k+2)} |P \del F|
\end{align*}
points in $P \del F$ have deficit at least $\lfloor \frac{Cn}{d} \rfloor$. It is easily seen (by considering integrality) that a point $p \in P$ has deficit less than $\lfloor \frac{Cn}{d} \rfloor$ if and only if it has degree at least $(1 - \frac{C}{d})n$. Hence at least $(1 - \frac{2d}{C(k+2)})_+|P \del F|$ points in $P \del F$ have degree at least $(1 - \frac{C}{d})n$.
\end{proof}

\subsection{Incidences and line lengths under nonconcentration}
\label{subsec:nonconcentration}

The results in this subsection use a nonconcentration hypothesis.
Point sets satisfying this nonconcentration hypothesis can be scaled to nearly isotropic position by Lemma~\ref{lem:frame} even without contracting a subspace, so Lemma~\ref{lem:quotient} plays no role in these proofs.

The incidence inequality of de Borbon and Panov
\cite[Corollary~4.1]{deBorbonPanov} and \cite[Theorem 6.29]{deBorbonPanovMY}, recovered in
Corollary~\ref{cor:nonconcentratedincidences}, generalizes
Langer's planar inequality~\cite{Langer}, which is stronger
than Hirzebruch's inequality~\cite{Hirzebruch1986}.
Finding an elementary proof of Hirzebruch's inequality
was explicitly posed as a problem by Brass, Moser, and
Pach~\cite[p.~315, Problem~7]{BrassMoserPach}.
The proof given here is not essentially different from that in \cite[Corollary~4.1]{deBorbonPanov}: both combine isotropic
normalization with rank-two frame energy estimates.

\begin{corollary}[An incidence bound under nonconcentration]
\label{cor:nonconcentratedincidences}
For $d \ge 2$ and $n\geq 1$, let $P\subset\C^d$ consist of $n$ distinct points. 
Suppose that
\begin{equation}\label{eq:uniformnonconcentration}
M_j(P)\le\frac{j+1}{d+1}n
\qquad(0\le j<d).
\end{equation}
Then
\begin{equation}\label{eq:nonconcentratedincidences}
\sum_{L\in\cL(P)}s_L
\ge\frac{d-1}{d+1}n^2+n.
\end{equation}
\end{corollary}

\begin{proof}
Use the construction in the proof of Lemma~\ref{lem:common}
with $k=d-1$, so that
\[
T\coloneq 1+\frac{n-d-1}{d+1}=\frac{n}{d+1}.
\]
For the homogeneous lifts $v_p\coloneq(p,1)\in\C^{d+1}$, hypothesis
\eqref{eq:uniformnonconcentration} gives
\[
|\{p\in P:v_p\in W\}|\le T\dim W
\]
for every linear subspace $W\subseteq\C^{d+1}$.
Consequently, the minimum-dimensional maximizer of
$|\{p\in P:v_p\in W\}|-T\dim W$ is $W=0$.
Thus the construction takes $F=\varnothing$, and
\eqref{eq:common} gives
\begin{equation}\label{eq:nonconcentratedcost}
\sum_{L\in\cL(P)}s_L(s_L-2)
\le\frac{2n(n-d-1)}{d+1}.
\end{equation}
Subtracting this inequality from the pair-counting identity
\[
\sum_{L\in\cL(P)}s_L(s_L-1)=n(n-1)
\]
proves \eqref{eq:nonconcentratedincidences}.
\end{proof}

The following consequence of Corollary~\ref{cor:nonconcentratedincidences} restates Corollary~\ref{cor:introordinary}.

\begin{corollary}[Ordinary lines under nonconcentration]
\label{cor:nonconcentratedordinary}
Under the hypotheses of
Corollary~\ref{cor:nonconcentratedincidences}, if $d\ge3$, then
the number $t_2$ of ordinary lines satisfies
\[
t_2\ge\frac{d-3}{2(d+1)}n^2+\frac32n.
\]
In particular, $t_2\ge3n/2$ when $d=3$, and the bound is
quadratic in $n$ for every fixed $d\ge4$.
\end{corollary}

\begin{proof}
Since $\binom{s}{2}\le s(s-2)$ for $s\ge3$,
\eqref{eq:nonconcentratedcost} implies
\[
\binom n2-t_2
=\sum_{L:\,s_L\ge3}\binom{s_L}{2}
\le\sum_{L\in\cL(P)}s_L(s_L-2)
\le\frac{2n(n-d-1)}{d+1}.
\]
Rearranging gives the stated bound.
\end{proof}

Let
\[
\overline s\coloneq \frac{1}{|\cL(P)|}\sum_{L\in\cL(P)}s_L
\]
be the average length of lines spanned by $P$.

The following corollary applies Corollary~\ref{cor:nonconcentratedincidences} to
the average length of a spanned line. This question belongs
to the study of average flat sizes in representable matroids;
see Campbell, Geelen, and Kroeker~\cite{CGK} and Campbell,
Kroeker, and Lund~\cite{CKL}. It also has a separate connection
to bounded negativity: Bauer et al.\ \cite{BauerEtAl} study the average intersection multiplicity
of a line arrangement, which is the same quantity under
projective duality.

The bounds below have exact equality examples, discussed in
Section~\ref{sec:constructions}. For fixed dimension, their
asymptotic optimality as $n\to\infty$ remains open.
In particular, Campbell, Geelen, and Kroeker
\cite[Conjecture~2.1]{CGK} predict an upper bound of $3+o(1)$
for $n$ noncollinear points in the complex plane.
In higher dimensions, the Dowling examples show that the leading
correction $4/(3d)$ is optimal as $d\to\infty$.

\begin{corollary}[Average line length under nonconcentration]
\label{cor:nonconcentratedaverage}
Under the hypotheses of
Corollary~\ref{cor:nonconcentratedincidences}, if $d\ge3$, then
\[
\overline s
\le
\frac{6\bigl((d-1)n+d+1\bigr)}
     {(3d-5)n+5(d+1)}
<
2+\frac{4}{3d-5}.
\]
If $d=2$, then
\[
\overline s
\le
\min\left\{
\frac{6(n+3)}{n+15},
\frac{4(n+3)}{n+7}
\right\}
<4.
\]
\end{corollary}

\begin{proof}
Put $b\coloneq|\cL(P)|$ and
$I\coloneq\sum_{L\in\cL(P)}s_L$.
For $k\in\{2,3\}$, integrality gives
$(s_L-k)(s_L-k-1)\ge0$.
Summing and using the pair-counting identity
$\sum_L s_L(s_L-1)=n(n-1)$ and
\eqref{eq:nonconcentratedincidences}, we obtain
\[
\frac{k(k+1)}{\overline s}
=\frac{k(k+1)b}{I}
\ge 2k-\frac{n(n-1)}{I}
\ge 2k-\frac{(d+1)(n-1)}{(d-1)n+d+1}.
\]
Taking $k=2$ gives the bound for $d\ge3$ and the first
planar bound; taking $k=3$ gives the second planar bound.
\end{proof}

\section{Reflection arrangements and sharpness}
\label{sec:constructions}

A complex reflection is a nonidentity linear transformation
that fixes a hyperplane pointwise. Let $W$ be a finite irreducible reflection group acting on $\mathbb{C}^r$.
Its reflection arrangement $\mathcal A(W)$ consists
of its reflecting hyperplanes. The corresponding projective
configuration $P_W$ consists of their normal lines. Choosing
an affine chart containing $P_W$ gives a point set of affine
dimension $d=r-1$, with the same incidence structure.

Two infinite families arise from the cyclic Dowling geometries
introduced by Dowling~\cite{Dowling}. Write
$\mu_t\coloneq\{\zeta\in\C:\zeta^t=1\}$. The cyclic Dowling
configuration is
\[
Q_r(t)=
\{[e_i]:1\le i\le r\}
\cup
\{[e_i-\zeta e_j]:1\le i<j\le r,\ \zeta\in\mu_t\},
\]
and the jointless cyclic Dowling configuration $Q'_r(t)$ is
obtained by deleting the coordinate points. Thus
\[
|Q_r(t)|=r+t\binom r2,
\qquad
|Q'_r(t)|=t\binom r2.
\]
In rank at least three, the two families $Q_r(t)$, $t\ge1$,
and $Q'_r(t)$, $t\ge2$, together with the fifteen exceptional
arrangements $G_{23},\ldots,G_{37}$, exhaust the irreducible
reflection arrangements. The exceptional ranks range from
three to eight. This follows from the classification of
Shephard and Todd~\cite[Table~VII]{ShephardTodd}.

Two general facts supply most of the sharpness statements
below. If $P=P_W$ has $n$ points, then
\[
M_j(P)<\frac{j+1}{r}n
\qquad(0\le j<r-1),
\]
and every point satisfies
\begin{equation}\label{eq:reflectiondegree}
\mathrm d_P(p)=1+\frac{r-2}{r}n.
\end{equation}
These are de Borbon--Panov
\cite[Lemmas~7.25 and~7.26]{deBorbonPanovMY};
the second is the hyperplane-restriction formula of
Orlik and Terao~\cite[Theorem~6.97]{OrlikTerao}.

Equation~\eqref{eq:reflectiondegree} shows that every
irreducible reflection configuration attains the total-degree
bound in Theorem~\ref{thm:introdegree}.
Letting $t$ grow in either Dowling family gives arbitrarily
large equality examples in every fixed dimension $d\ge2$.
For ordinary degree, counting the other points on lines
through $p$ gives
\begin{equation}\label{eq:reflectionordinarydegree}
o_P(p)
=
n+3-\frac{4n}{r}
+
\sum_{\substack{L\ni p\\s_L\ge4}}(s_L-3).
\end{equation}
Consequently, the maximum ordinary degree equals the theorem's
lower bound precisely when every spanned line contains at
most three points.

For the Dowling configurations, the possible line sizes follow
directly from their defining vectors. For the exceptional
configurations, the corresponding information follows from
the tabulated intersection types; see
Muraleedaran and Taylor~\cite[Tables~4--14]{MuraleedaranTaylor}.
Here a line of $s$ normal points corresponds to a
codimension-two intersection contained in $s$ reflecting
hyperplanes. Consequently, among irreducible reflection
configurations of rank at least three, every spanned line
contains at most three points precisely for
\[
Q_r(1),\qquad Q'_r(2),\qquad Q'_r(3),
\]
and
\[
G_{33},\quad G_{34},\quad
G_{35}=E_6,\quad G_{36}=E_7,\quad G_{37}=E_8.
\]
Thus the ordinary-degree bound is also attained in every
dimension $d\ge2$.

The same general facts give a short proof that the positivity
threshold in Theorem~\ref{thm:intromany} is sharp.
Fix $d\ge2$, $0\le k<d$, and
\[
1\le C<\frac{2d}{k+2}.
\]
Put $s=k+2$. Take an affine realization of $Q_s(t)$, with
$m$ points, and add $N=d-k-1$ points whose homogeneous
vectors are independent modulo the span of the core.
The resulting set has affine dimension $d$ and $n=m+N$
points. Every line involving an added point is ordinary,
so each core point has degree
\[
n+1-\frac{2m}{s}.
\]
Since $m/n\to1$, no core point has degree at least
$(1-C/d)n$ once $t$ is sufficiently large. The
nonconcentration inequality for the core implies that every
$k$-flat misses at least $m/s$ core points. Thus the proportion
of high-degree points outside any such flat is at most
\[
\frac{Ns}{m}\longrightarrow0.
\]
This proves sharpness of the positivity threshold for every
$k$, without asserting sharpness of the guaranteed proportion
above that threshold.

Summing \eqref{eq:reflectiondegree} gives
\[
\sum_{L\in\cL(P)}s_L
=
\frac{r-2}{r}n^2+n.
\]
Thus every irreducible reflection configuration satisfies
the strict nonconcentration hypothesis and attains equality
in Corollary~\ref{cor:nonconcentratedincidences}, as well
as in the equivalent quadratic estimate
\eqref{eq:nonconcentratedcost}.

Summing \eqref{eq:reflectionordinarydegree} similarly gives
\[
t_2(P)
=
\frac{r-4}{2r}n^2+\frac32n
+
\frac12
\sum_{\substack{L\in\cL(P)\\s_L\ge4}}s_L(s_L-3).
\]
For $d\ge3$, Corollary~\ref{cor:introordinary}
is therefore an equality precisely for the configurations
listed above with all line sizes at most three.
The jointless family also gives asymptotic examples.
For $t\ge3$, its coordinate lines contain $t$ points and
all other lines contain two or three points. Substitution
in the preceding identity yields
\[
t_2\bigl(Q'_{d+1}(t)\bigr)
=
\left(
\frac{d-3}{2(d+1)}
+\frac{1}{d(d+1)}
\right)n^2.
\]
Hence, for fixed $d$, the coefficient of $n^2$ in the
corollary can be improved by at most $1/[d(d+1)]$.
For $r=4$ and $t\ge3$, the standard description of Dowling
flats~\cite{Dowling} also gives
\[
M_2\bigl(Q'_4(t)\bigr)=3t=\frac n2.
\]

For $d\ge3$, the proof of
Corollary~\ref{cor:nonconcentratedaverage} has equality in its
exact bound precisely when the incidence bound is an equality
and every line contains two or three points. Its reflection
equality examples are therefore exactly the same as those
for Corollary~\ref{cor:introordinary}.
In dimension two, the first branch of the bound is attained
by $A_3=Q_3(1)$ and the dual Hesse configuration $Q'_3(3)$,
with six and nine points. The second branch is attained
by $Q'_3(4)$ and the Klein configuration $G_{24}$, with
twelve and twenty-one points. The required line multiplicities
for the latter examples are recorded by
Pokora~\cite[Remark~2.5]{PokoraOrbifold}.

Among irreducible reflection configurations, the largest
average line lengths in ranks three and four are
\[
\overline s(P_{G_{27}})=\frac{240}{67}\approx3.582,
\qquad
\overline s(P_{G_{29}})=\frac{84}{31}\approx2.710.
\]
The first is attained by Wiman's configuration, with
forty-five points, and the second by $G_{29}$, with forty
points. These comparisons follow from the two infinite
families and the exceptional intersection data in
\cite[Tables~4--9]{MuraleedaranTaylor}.
Wiman's configuration is also the largest rank-three example
known to Campbell, Geelen, and Kroeker
\cite[Section~2]{CGK}.

The asymptotic average can be obtained without enumerating
individual line types. The Whitney-number formula for
Dowling lattices gives
\[
|\cL(Q_r(t))|
=
\left(\binom r3+3\binom r4\right)t^2+O_r(t);
\]
see Merca~\cite[Corollary~3]{MercaWhitney}.
Deleting the $r$ coordinate points changes this count by
only $O_r(t)$, since each deleted point lies on $O_r(t)$
lines. Combining this with the incidence identity gives,
for either Dowling family,
\[
\lim_{t\to\infty}\overline s
=
2+\frac{4}{3d-2}.
\]
If $f(d)$ denotes the optimal asymptotic correction in a
bound
\[
\overline s\le2+f(d)+o_{n\to\infty}(1)
\]
under uniform nonconcentration, then
Corollary~\ref{cor:nonconcentratedaverage} implies
\[
\frac{4}{3d-2}\le f(d)\le\frac{4}{3d-5}
\qquad(d\ge3).
\]
In particular,
\[
f(d)=\frac{4}{3d}+O(d^{-2}).
\]
For $d=2$, the Dowling families approach average line length
three.

Finally, reflection configurations give equality examples
for the excluded-line-minor bound in
Theorem~\ref{thm:introminor}. The cyclic Dowling geometry $Q_r(t)$
has no $U_{2,t+3}$-minor and has $r+t\binom r2$ points,
so it attains equality for $\ell=t+3$.
These standard properties are recalled by
Geelen, Nelson, and Walsh~\cite[Introduction]{GNW}.

In rank three, every irreducible reflection configuration
attains the excluded-line-minor bound. Indeed, every point
has degree $D=n/3+1$, so every simplified point contraction
has $D$ points. Every rank-two restriction embeds into a
point contraction, by contracting a point outside its span.
Hence the configuration has no $U_{2,D+1}$-minor and attains
equality with $\ell=D+1=n/3+2$.

In rank four, $G_{32}$ supplies a further example.
Its configuration has forty points, and every simplified
point contraction is isomorphic to the twenty-one-point
configuration of $G_{26}$. This is the
arrangement-restriction identification recorded by
Amend, Hoge, and R\"ohrle
\cite[Lemma~3.3(vii) and Table~3]{AmendHogeRoehrle},
following Orlik and Terao.
The rank-three argument therefore shows that its largest
simple rank-two minor has eight points. Hence it excludes
$U_{2,9}$ and attains
\[
40=4+(9-3)\binom42.
\]

\section{Matroid density bound and matrix column bounds}\label{sec:minors}

In this section we separately consider $\Delta$-submodular matrices as well as the more general class of (complex) matrices of bounded circuit-imbalance. In both cases, our density bounds are derived using Theorem~\ref{thm:introminor}, our density bound for $U_{2,\ell}$-minor-free complex-representable matroids, which we restate here for convenience.

\begin{theorem}\label{thm:minor}
Let $M$ be a simple rank-$r$ complex-representable matroid. For any integer $\ell \geq 3$, if $M$ has no $U_{2,\ell}$-minor, then $|E(M)| \leq r + (\ell - 3)\binom r2$. 
\end{theorem}

\begin{proof}
Clearly we may assume that $r \geq 2$. Moreover, if $M$ is $U_{2,3}$-minor-free, then $M$ is a free matroid; so we may assume that $\ell \geq 4$. Now we proceed by induction on $r \geq 2$. Throughout, let $n$ denote the number of elements in $M$. If $r = 2$, then $n \leq \ell - 1 = r + (\ell - 3)\binom r2$. Now let $r \geq 3$. By Theorem~\ref{thm:introdegree}, there exists $e \in E(M)$ contained in at least $n + 1 - \frac{2n}{r}$ lines of $M$. Let $N$ be a simplification of $M/e$; that is, $N$ is obtained from $M/e$ by removing all but exactly one element from each set elements in parallel (so the elements of $N$ correspond exactly to the lines of $M$ through $e$). Then, by induction,
\begin{align*}
    n + 1 - \frac{2n}{r} \leq |E(N)| \leq r-1 + (\ell -3)\binom{r-1}{2},
\end{align*}
and the desired bound follows by rearranging.
\end{proof}

\subsection{Column bounds for \texorpdfstring{$\Delta$}{Delta}-submodular matrices}\label{subsec:delta}

Let $A\in\Z^{r\times n}$ have full row rank $r\ge1$ and
$\Delta_A\coloneq\max_{|B|=r}|\det A_B|\le\Delta$, where $\Delta\ge1$
is an integer. This is called $\Delta$-modular in some references
and $\Delta$-submodular in others. It only bounds maximal minors;
no bound on smaller minors is assumed. 

Combined with an arithmetic reduction of \cite[Proposition~8.6.1]{GNW}, Theorem~\ref{thm:introminor} immediately implies the following upper bound on the number of columns in a $\Delta$-submodular matrix with nonparallel columns.

\begin{corollary}[Nonparallel integer columns]\label{cor:delta}
Let $A \in \Z^{r \times n}$ be a rank-$r$ $\Delta$-submodular matrix. If the columns are nonzero and pairwise nonparallel, then
\begin{equation}\label{eq:delta}
n\le r+(2\Delta-1)\binom r2.
\end{equation}
More precisely, $2\Delta-1$ may be replaced by $p-1$, where $p$
is any prime greater than $\Delta$.
\end{corollary}
\begin{proof}
For $r=1$, nonzero pairwise nonparallel columns force $n=1$, so the
claim is immediate. Assume $r\ge2$.
Reduction modulo $p>\Delta$ preserves exactly which maximal minors
are nonzero, hence preserves the represented matroid. It is thus
representable over both $\C$ and $\mathbb F_p$. Its simple rank-two
minors have at most $p+1$ elements, so apply
Theorem~\ref{thm:introminor} with $\ell=p+2$. A prime
$\Delta<p\le2\Delta$ exists, including $p=2$ for $\Delta=1$.
\end{proof}

Next, we apply Corollary~\ref{cor:delta} to obtain a general transfer lemma from bounds on nonparallel columns to bounds on distinct columns.
The proof groups the nonzero columns, initially taken up to sign, by
their \emph{content}, the greatest common divisor of their entries.
Within a group of content $q$, dividing every column by $q$ produces
primitive, pairwise nonparallel columns, so
Corollary~\ref{cor:delta} applies. The key saving is that if this group
has rank $s$, any basis chosen from it contributes a factor $q^s$
to a full determinant. After division, the effective determinant
bound is therefore at most $\Delta/q^s$. For $q>1$, this exponential
saving in the rank controls the quadratic rank dependence in the
preceding corollary: summing the bound from Corollary~\ref{cor:delta} over all such groups gives at most
$3(\Delta-1)$ columns, independently of the ambient rank. Any given bound for nonparallel columns can be used to control the primitive columns with content $1$, which supply the main term of the output bound for nonzero columns, distinct up to sign.
Finally, allowing both signs doubles the bound,
and allowing a zero column adds one.

\begin{lemma}[From nonparallel to distinct columns]
\label{lem:distincttransfer}
Fix integers $r,\Delta\ge1$, and let $F(r,\Delta)$ be an upper
bound for the number of nonzero pairwise nonparallel columns
in a rank-$r$, $\Delta$-submodular integer matrix.
Let $A\in\Z^{r\times n}$ be a rank-$r$ $\Delta$-submodular matrix.
If its columns are nonzero and distinct up to sign, then
\[
n\le F(r,\Delta)+3(\Delta-1).
\]
If its columns are pairwise distinct, allowing both signs and zero,
then
\[
n\le2F(r,\Delta)+6(\Delta-1)+1.
\]
\end{lemma}
\begin{proof}
First assume that the columns are nonzero and distinct up to sign.
Divide each column by its content and retain one primitive
representative from each direction. The resulting integer matrix
has rank $r$ and maximal minors bounded by $\Delta$, since each
determinant is an original determinant divided by a product of
contents. Thus the number $N$ of directions satisfies
$N\le F(r,\Delta)$.

Each direction contains at most one original primitive column.
Consequently,
\[
n-N\le\sum_{q=2}^{\Delta}m_q,
\]
where $m_q$ counts the columns of content $q$. The upper limit is
valid because every nonzero column extends to a basis, whose
determinant is divisible by its content.

Choose $s$ linearly independent columns of the layer and let $C$
be the matrix they form. Applying
Schrijver~\cite[Corollary~4.3b]{Schrijver} to $C^T$ and transposing
gives a unimodular integer matrix $U$ such that $UC$ has zero
last $r-s$ rows and a nonsingular block in its first $s$ rows.
Thus $U$ sends $V$ to the first $s$ coordinate directions.
Moreover, $U$ preserves contents: since both $U$ and $U^{-1}$
are integral, an integer divides every entry of a column $a$
if and only if it divides every entry of $Ua$.

Consider a nonempty content-$q$ layer of rank $s$.
A unimodular integer row change sends this layer to the first $s$ coordinate directions.
In detail, choose $s$ linearly independent columns of the layer and let $C$
be the matrix they form. Applying
Schrijver~\cite[Corollary~4.3b]{Schrijver} to $C^T$ and transposing
gives a unimodular integer matrix $U$ such that $UC$ has zero
last $r-s$ rows and a nonsingular block in its first $s$ rows.
Thus $U$ sends $V$ to the first $s$ coordinate directions.
Moreover, $U$ preserves contents: since both $U$ and $U^{-1}$
are integral, an integer divides every entry of a column $a$
if and only if it divides every entry of $Ua$.

Deleting the zero
rows and dividing by $q$ gives a rank-$s$ integer matrix $B_q$
with primitive, pairwise nonparallel columns.
Choose $r-s$ original columns complementary to the layer's span.
For every $s$-column basis $(B_q)_I$, the corresponding original
determinant has the form
\[
q^s\det((B_q)_I)\det H,
\]
where $H$ is a fixed nonsingular integer matrix, with $\det H=1$
when $s=r$. Hence the maximal minors of $B_q$ are bounded by
$\lfloor\Delta/q^s\rfloor$.

Corollary~\ref{cor:delta} gives
\[
m_q\le
s+\left(2\left\lfloor\frac{\Delta}{q^s}\right\rfloor-1\right)
\binom s2
\le1+\frac{\Delta s(s-1)}{q^s}
\le1+\frac{2\Delta}{q(q-1)}.
\]
Here $s-\binom s2\le1$, and
$\binom s2(q-1)\le q^{s-1}$ follows by expanding
$(1+(q-1))^{s-1}$.
Summing over $2\le q\le\Delta$, including empty layers, yields
\[
\sum_{q=2}^{\Delta}m_q
\le(\Delta-1)+2\Delta\sum_{q=2}^{\Delta}\frac1{q(q-1)}
=3(\Delta-1).
\]
This proves the first assertion. For arbitrary distinct columns,
retain one representative from each pair $\{a,-a\}$ and omit zero.
Doubling the first bound and adding one proves the second.
\end{proof}

Applying Lemma~\ref{lem:distincttransfer} with Corollary~\ref{cor:delta} gives the following corollary. Its first assertion restates Theorem~\ref{thm:introcolumns}; the second allows both signs and a zero column.

\begin{corollary}[Distinct integer columns]\label{cor:distinct}
Let $A \in \Z^{r \times n}$ be a rank-$r$ $\Delta$-submodular matrix. If the columns of $A$ are nonzero and distinct up to sign, then
\begin{equation}\label{eq:differing}
n\le(2\Delta-1)\binom{r+1}{2}.
\end{equation}
If they are merely pairwise distinct, allowing both signs and a zero
column, then
\begin{equation}\label{eq:distinct}
n\le2(2\Delta-1)\binom{r+1}{2}+1=O(r^2\Delta).
\end{equation}
\end{corollary}
\begin{proof}
First assume that the columns are nonzero and distinct up to sign.
For $r=1$, there are at most $\Delta$ columns, so the assertion
holds. For $r\ge2$, apply Lemma~\ref{lem:distincttransfer} with
the bound from Corollary~\ref{cor:delta},
\[
F(r,\Delta)\coloneq r+(2\Delta-1)\binom r2.
\]
This gives
\[
n\le r+(2\Delta-1)\binom r2+3(\Delta-1)
\le(2\Delta-1)\binom{r+1}{2},
\]
where the difference in the last inequality is
$(\Delta-1)(2r-3)\ge0$. This proves \eqref{eq:differing}.
For arbitrary distinct columns, retaining one representative from
each sign pair and omitting zero, then doubling the bound and
adding one, proves \eqref{eq:distinct}.
\end{proof}

\subsection{Column bounds from circuit imbalance}\label{sec:circuitimbalance}

For a real or complex matrix $A$, a \emph{circuit vector} is a
nonzero vector in $\ker A$ with inclusion-minimal support. Define
\[
\kappa_A\coloneq\max\left\{
\left|\frac{g_i}{g_j}\right|:
g\text{ is a circuit vector},\ i,j\in\operatorname{supp}g
\right\}.
\]
For a matrix without dependencies, set $\kappa_A\coloneq 1$. For full row
rank, the equivalent determinant formula is
\begin{equation}\label{eq:kappadet}
\kappa_A=\max\left\{1,
\left|\frac{\det A_{B-i+j}}{\det A_B}\right|:
A_B\text{ invertible},\ i\in B,\ j\notin B
\right\}.
\end{equation}
This standard formula is also stated in \cite[Introduction]{DEPRS}.
It follows from Cramer's rule and extension of a circuit minus one
element to a basis. The determinants in a ratio differ by \emph{one}
column; this is not the ratio of the largest to smallest arbitrary
nonzero maximal minors. Invertible row operations preserve $\kappa$;
arbitrary individual column rescalings need not.

For a full-row-rank integer matrix $A$, every nonzero maximal
minor has absolute value at least one, so \eqref{eq:kappadet} gives
$\kappa_A\le\Delta_A\le\Delta$ whenever $A$ is $\Delta$-submodular. Hence, column bounds for $\Delta$-submodular matrices follow immediately from corresponding bounds for real matrices with bounded circuit imbalance.
Conversely, bounded circuit imbalance does not bound $\Delta_A$ for a given integer representation: multiplying a rank-$r$
integer matrix by a positive integer $q$ preserves $\kappa_A$ but
multiplies $\Delta_A$ by $q^r$.

We use the reduction from circuit imbalance to excluded line minors
in \cite[Section~3.3]{DEPRS}.
We include all details for the reader's convenience.
We first prove a column bound for rank two, which slightly improves the corresponding bound of \cite[Lemma 3.4]{DEPRS}.
For higher rank, we use the rank two bound to show that there is no $U_{2,\ell}$ minor for $\ell$ sufficiently large.
This enables us to apply Theorem~\ref{thm:introminor} to bound the number of columns.

The following monotonicity is \cite[Proposition~3.17]{ENV}, invoked
as \cite[Lemma~3.3]{DEPRS}. We give a direct circuit-lifting proof
valid over both $\R$ and $\C$.

\begin{lemma}[Minor monotonicity]\label{lem:kappaminor}
Deletion and contraction admit matrix representations whose circuit
imbalance does not exceed that of the original matrix. Simplification
can be performed by deletion, retaining the inherited column scalings.
\end{lemma}
\begin{proof}
Deletion only restricts the possible circuit vectors. For contraction
of a set $C$, the kernel of the quotient representation is the
coordinate projection of $\ker A$ onto the remaining index set $J$.
Let $h$ be a circuit of that projected kernel, and choose a lift
$g\in\ker A$ with $g_J=h$ and minimal support among such lifts.

If $g$ is not a circuit, choose a nonzero kernel vector $f$ supported
properly inside its support. If $f_J=0$, subtracting a suitable
multiple of $f$ from $g$ cancels a contracted coordinate and gives
a smaller lift. If $f_J\ne0$, circuit minimality makes $f_J$ a scalar
multiple of $h$ (the dependency space on a minimal support is
one-dimensional), and a rescaling of $f$ is a smaller lift. Both
contradict the choice of $g$. Thus $g$ is a circuit and all nonzero
coefficient ratios in $h$ occur in $g$. Finally, remove loops and
extra parallel representatives by deletion.
\end{proof}

\begin{lemma}[Rank-two column bounds]\label{lem:rank2kappa}
A rank-two matrix with $m$ nonzero pairwise nonparallel columns and
circuit imbalance at most $\kappa$ satisfies
\[
m\le2\pi\kappa\quad\text{over }\R,
\qquad
m\le16\kappa^2\quad\text{over }\C.
\]
\end{lemma}
\begin{proof}
We refine the maximum-determinant normalization and packing argument
of \cite[Lemma~3.4]{DEPRS}. The normalization works over either field.
Choose a maximum-modulus determinant basis and send it to $e_1,e_2$.
Every coordinate of every column is now at most $1$ in modulus,
so every column has norm at most $\sqrt2$. For distinct columns
$v,w$, each entry of $[v\ w]^{-1}$ is at most $\kappa$ in modulus:
its columns express $e_1,e_2$ in the basis $v,w$, so use
\eqref{eq:kappadet}, or the trivial coefficients when an axis is
itself in the pair. The inverse formula implies
\[
|\det(v,w)|\ge\frac{\max\{\|v\|,\|w\|\}}{\sqrt2\kappa},
\qquad
\frac{|\det(v,w)|}{\|v\|\|w\|}\ge\frac1{2\kappa}.
\]
Over $\R$, the projective directions therefore have angular separation
at least $\theta\coloneq\arcsin(1/(2\kappa))$. The projective circle has
length $\pi$, so
\[
m\le\frac\pi{\arcsin(1/(2\kappa))}\le2\pi\kappa.
\]
Over $\C$, for a unit vector $z=(z_1,z_2)$, associate the point
\[
q(z)\coloneq(2\operatorname{Re}(z_1\overline z_2),\ 2\operatorname{Im}(z_1\overline z_2),
|z_1|^2-|z_2|^2)\in S^2.
\]
It depends only on the projective direction and satisfies
$\|q(z)-q(w)\|=2|\det(z,w)|$. Thus our $m$ sphere points have
Euclidean separation at least $1/\kappa$. Spherical caps of angular
radius $\theta\coloneq\arcsin(1/(2\kappa))$ have disjoint interiors.
Comparing their areas with $4\pi$ gives
\[
m\le\frac2{1-\sqrt{1-1/(4\kappa^2)}}\le16\kappa^2.
\]
\end{proof}

Over $\R$, our sharper norm estimates improve $4\pi\kappa$ in
\cite[Lemma~3.4]{DEPRS} to $2\pi\kappa$. Over $\C$, their proof of
Corollary~1.6 uses successive pigeonhole arguments on phases and
coordinate moduli to obtain $O(\kappa^2)$. The sphere representation
above gives an explicit constant by the same projective packing
principle.

\begin{corollary}[Column bounds from circuit imbalance]\label{cor:kappa}
Let $A$ have rank $r\ge2$ and $n$ nonzero pairwise nonparallel
columns. Over $\R$,
\begin{equation}\label{eq:realkappa}
n\le r+\bigl(\lfloor2\pi\kappa_A\rfloor-2\bigr)\binom r2
\le\pi r^2\kappa_A.
\end{equation}
Over $\C$,
\begin{equation}\label{eq:complexkappa}
n\le r+\bigl(\lfloor16\kappa_A^2\rfloor-2\bigr)\binom r2
\le8r^2\kappa_A^2.
\end{equation}
\end{corollary}
\begin{proof}
Lemmas~\ref{lem:kappaminor} and~\ref{lem:rank2kappa} bound the size
of every simple rank-two minor. Apply Theorem~\ref{thm:introminor}.
\end{proof}

Theorem~1.1 of \cite{DEPRS} gives $n\le\pi r^4\kappa_A$ over
$\R$, and Corollary~1.6 gives $O(r^4\kappa_A^2)$ over $\C$;
these statements use row dimension at least four. The estimates above replace the
quartic rank dependence by quadratic dependence uniformly in both
parameters. This asymptotic improvement comes from Theorem~\ref{thm:introminor}.

\section{Ordinary lines and the conjecture of Basit et al.}\label{sec:ordinary}

Basit--Dvir--Saraf--Wolf \cite[after Theorem~1.10]{BDSW} ask whether
a quadratic number of ordinary lines follows when no affine plane
contains more than $n/2$ points. This is an unnumbered conjectural
statement in the corrected version of their paper. Their quadratic
result assumes nonconcentration in three-dimensional affine spaces;
their three-dimensional results give linear bounds and refinements
involving longer lines. De Zeeuw \cite{deZeeuw} proves the
plane-nonconcentration statement over $\R$, for every fixed plane
fraction less than one, using projection and Beck's theorem. We prove
the complex statement by supplementing the second-moment estimate of
Section~\ref{sec:frames} with a fourth-moment bound.

\subsection{The second-moment threshold and proof idea}

The second-moment estimate applied in Section~\ref{sec:common} cannot give a quadratic lower bound on the number of ordinary lines in affine dimension three, which corresponds to homogeneous rank four.
To see the threshold,
consider an isotropically normalized configuration of homogeneous
rank $D$, so $S=(n/D)I_D$. Its total squared correlation is
$E=\tr S^2-n=n^2/D-n$. 
On the other hand, the rank-two estimate \eqref{eq:localenergy} shows that the average squared correlation contributed by each ordered pair of points in a nonordinary line is at least $1/4$.
If $1/4 > 1/D$, there must be quadratically many ordinary pairs, because if every pair of points was in a nonordinary line, then the local contribution from nonordinary lines would be too large.
To bypass this problem, we use a centered fourth-moment estimate.

We work directly with the homogeneous vectors
$v_i\coloneq(p_i,1)\in\C^4$, without taking a quotient. Distinct points
give nonparallel vectors, and an invertible linear transformation
followed by normalization preserves this property. For each original spanned line $L$, let
$I_L\coloneq\{i:p_i\in L\}$. These sets form the edges of a linear
hypergraph $\mathcal H$ on $[n]$, in which every pair of distinct
vertices belongs to exactly one edge. The vectors indexed by each
edge span a subspace of dimension exactly two. For the resulting
unit vectors, put $S\coloneq\sum_i u_i u_i^*$ and define
\begin{align*}
X&\coloneq\sum_{i\ne j}\left(|\langle u_i,u_j\rangle|^2-\tfrac14\right)
=\tr S^2-n-\frac{n(n-1)}4,\\
Y&\coloneq\sum_{i\ne j}\left(|\langle u_i,u_j\rangle|^2-\tfrac14\right)^2.
\end{align*}

We first prove a global estimate that gives a quadratic lower bound on $Y$.
Then, we prove a local estimate that implies that, in the isotropic case, this lower bound can't be accounted for by pairs of points in nonordinary lines.
An easy counting argument shows that there must be many ordinary lines if a single plane contains at least $7n/12$ points.
Otherwise, if no line contains more than half of the points, we obtain a near-isotropic scaling of the points using Lemma~\ref{lem:frame}.
If some line contains more than $n/2$ (but at most $7n/12$) points, then local accounting on the single heavy line more than compensates for the fact that our scaling isn't quite isotropic.

\subsection{Global and local fourth-moment bounds}

We prove a global fourth moment estimate by applying Cauchy--Schwarz to get a second moment estimate for the centered projectors $u_i u_i^*-I_4/4$.

\begin{lemma}[Centered projectors in $\C^4$]\label{lem:projectors}
For unit vectors $u_1,\ldots,u_n\in\C^4$, the quantity $Y$
defined above satisfies
\begin{equation}\label{eq:projectorfourth}
Y\ge\frac{3n^2}{80}-\frac{9n}{16}.
\end{equation}
\end{lemma}
\begin{proof}
The matrices $Q_i\coloneq u_i u_i^*-I_4/4$ lie in the real inner-product
space of traceless Hermitian $4\times4$ matrices, with inner product
$\langle A,B\rangle\coloneq\tr(AB)$.
A Hermitian $4\times4$ matrix has four real diagonal entries and six
freely chosen complex entries above the diagonal, giving $16$ real
parameters. The trace-zero condition imposes one independent real
linear constraint, so this space has real dimension $15$.
The Gram matrix $K$ of the $Q_i$ is real positive semidefinite, with
\[
K_{ij}=\tr(Q_iQ_j)=|\langle u_i,u_j\rangle|^2-\tfrac14,
\qquad K_{ii}=\tfrac34.
\]
Thus $\rank K\le15$ and $\tr K=3n/4$. Cauchy--Schwarz applied
to the nonzero eigenvalues of $K$ gives
$\tr K^2\ge(3n/4)^2/15=3n^2/80$.
Removing the $n$ diagonal contributions $(3/4)^2$ proves the bound.
\end{proof}

\begin{lemma}[Fourth moments for linear hypergraphs]\label{lem:fourth}
Let $u_1,\ldots,u_n\in\C^4$ be unit vectors, and let $\mathcal H$
be a linear hypergraph on $[n]$ such that every pair of distinct
vertices belongs to exactly one edge and
$\rank(\{u_i:i\in I\})=2$ for every $I\in\mathcal H$.
Let $b$ be the number of edges of size two, and put
\[
R\coloneq\frac38\sum_{\substack{I\in\mathcal H\\|I|\ge4}}|I|(|I|-4).
\]
Then
\begin{equation}\label{eq:fourth}
Y+R\le\frac{11}{4}X+\frac32b.
\end{equation}
\end{lemma}
\begin{proof}
For $I\in\mathcal H$, write
\begin{align*}
X_I&\coloneq\sum_{\substack{i,j\in I\\i\ne j}}
\left(|\langle u_i,u_j\rangle|^2-\tfrac14\right),\\
Y_I&\coloneq\sum_{\substack{i,j\in I\\i\ne j}}
\left(|\langle u_i,u_j\rangle|^2-\tfrac14\right)^2.
\end{align*}

If $|I|=2$, relabel its indices as $1,2$ and put $x\coloneq|\langle u_1,u_2\rangle|^2$.
Since $0\le x\le1$,
\begin{equation}\label{eq:local2}
Y_I-\frac{11}{4}X_I
=2x^2-\frac{13}{2}x+\frac32
\le\frac32.
\end{equation}

For $|I|\ge4$, the trace--rank bound \eqref{eq:localenergy} gives
$X_I\ge |I|(|I|-3)/4$. The inequality
\[
(x-\tfrac14)^2
\le\frac12(x-\tfrac14)+\frac3{16}
\qquad(0\le x\le1)
\]
gives $Y_I\le X_I/2+3|I|(|I|-1)/16$, and hence
\begin{equation}\label{eq:local4}
\frac{11}{4}X_I-Y_I
\ge\frac94X_I-\frac{3|I|(|I|-1)}{16}
\ge\frac38|I|(|I|-4).
\end{equation}

It remains to prove that, for $|I|=3$,
\begin{equation}\label{eq:local3}
Y_I\le\frac{11}{4}X_I.
\end{equation}
Relabel the indices as $1,2,3$ and put
\begin{gather*}
x\coloneq|\langle u_1,u_2\rangle|^2,\qquad
y\coloneq|\langle u_1,u_3\rangle|^2,\qquad
z\coloneq|\langle u_2,u_3\rangle|^2,\\
s\coloneq x+y+z,\qquad q\coloneq x^2+y^2+z^2.
\end{gather*}
The trace--rank bound \eqref{eq:localenergy} gives $s\ge3/4$.
The $3\times3$ Gram matrix has determinant zero because the vectors
indexed by this edge span a two-dimensional subspace. Its expansion gives
\[
1-s=-2\operatorname{Re}\bigl(\langle u_1,u_2\rangle
\langle u_2,u_3\rangle\langle u_3,u_1\rangle\bigr),
\]
and hence
\[
(1-s)^2\le4xyz.
\]
The desired inequality is equivalent to $13s-4q\ge9$.
If $s\ge1$, this follows from $q\le s$.

Fix $3/4\le s<1$ and maximize $q$ over nonnegative triples with
sum $s$ and $xyz\ge(1-s)^2/4$. All coordinates are positive.
At a maximizer the product constraint is an equality: otherwise
a small perturbation $(x,y)\mapsto(x+\varepsilon,y-\varepsilon)$,
with a suitable sign, increases $q$ while preserving feasibility.
Unless $x=y=z$, the constraint gradients are independent, and the
Lagrange multiplier equations give
\[
2x^2=\lambda x+\mu xyz,
\]
with analogous equations for $y,z$. Thus $x,y,z$ are roots of one
quadratic, so two coordinates are equal. This also holds when all
three are equal.

After relabeling, write $x=r^2$ and $y=z$, where $0<r<1$.
Since $s<1$, the product equality gives
\[
1-r^2-2y=2ry,
\qquad\text{so}\qquad y=z=\frac{1-r}{2}.
\]
Consequently,
\[
13s-4q-9=(1-r)(r+2)(2r-1)^2\ge0.
\]
Summing these estimates over the edges proves the lemma, since
each pair of distinct vertices belongs to exactly one edge.
\end{proof}

\subsection{Ordinary-line estimates and proof of the theorem}

\begin{lemma}[Ordinary lines crossing a flat]\label{lem:crossing}
If a flat contains $m$ of the $n$ points, then at least
\[
(n-m)\max\{0,2m-n+1\}
\]
ordinary lines have one point in the flat and one outside it.
\end{lemma}
\begin{proof}
Fix an outside point $q$. Its lines to the $m$ inside points are
distinct, because a line not contained in a flat meets it in at most
one point. Each nonordinary such line contains another outside
point, and different lines use different such witnesses. At most
$n-m-1$ lines are therefore nonordinary. Sum over $q$; an ordinary
crossing line is counted exactly once.
\end{proof}

Under the hypotheses of Theorem~\ref{thm:introordinary}, a plane containing
$m>3n/4$ points therefore gives
\[
t_2\ge(n-m)(2m-n+1)\ge\frac{1-\alpha}{2}n^2.
\]
If instead $7n/12<m\le3n/4$, put $x\coloneq m/n$. Then
\[
t_2\ge(1-x)(2x-1)n^2
\ge\frac5{72}n^2>\frac{n^2}{40},
\]
since the concave quadratic $(1-x)(2x-1)$ has minimum $5/72$
on $[7/12,3/4]$. Thus any plane containing more than $7n/12$
points already gives the conclusion of the theorem.

We now restate Theorem~\ref{thm:introordinary} and complete its proof.

\begin{theorem}[Ordinary lines under plane nonconcentration]\label{thm:ordinary}
Let $P\subset\C^3$ consist of $n\ge1$ distinct points, and suppose that
every affine plane contains at most $\alpha n$ of them, where
$0\le\alpha<1$. If $t_2(P)$ is the number of ordinary lines, then
\begin{equation}\label{eq:ordinarymain}
t_2(P)\ge
\min\left\{\frac1{40},\frac{1-\alpha}{2}\right\}n^2.
\end{equation}
\end{theorem}

\begin{proof}[Proof of Theorem~\ref{thm:introordinary}]
The hypothesis implies that $P$ spans $\C^3$, so $n\ge4$.
By the crossing-line estimates above, we may assume that every plane
contains at most $7n/12$ points. Let $\ell$ be the maximum number of
collinear points. Every line lies in a plane, so $\ell\le7n/12$.

Homogenize $P$, apply an invertible linear transformation in $\C^4$,
and normalize the images to unit vectors $u_i$. Write
\[
S\coloneq\sum_i u_i u_i^*,
\qquad
X\coloneq\tr S^2-n-\frac{n(n-1)}4.
\]
The edges $I_L$ associated with the original lines form a linear
hypergraph in which every pair of distinct vertices belongs to
exactly one edge. The vectors indexed by each edge have rank exactly
two, and the edges of size two correspond precisely to the ordinary lines.
Combining \eqref{eq:projectorfourth} and \eqref{eq:fourth} therefore
gives
\[
\frac32t_2
\ge\frac{3n^2}{80}-\frac{9n}{16}-\frac{11}{4}X
+\frac38\sum_{L:\ s_L\ge4}s_L(s_L-4).
\]
Equivalently,
\begin{equation}\label{eq:ordinaryX}
t_2\ge
\frac{n^2}{40}-\frac{3n}{8}-\frac{11}{6}X
+\frac14\sum_{L:\ s_L\ge4}s_L(s_L-4).
\end{equation}

Suppose first that $\ell\le n/2$. The homogeneous vectors satisfy
\eqref{eq:density} with $T=n/4$: subspaces of dimensions $1,2,3,4$
contain at most $1,n/2,7n/12,n$ vectors, respectively.
For every $\eta>0$, Lemma~\ref{lem:frame} provides a normalization
with $S\preceq(n/4+\eta)I_4$. Since $\tr S=n$, we have
\[
X\le n(n/4+\eta)-n-\frac{n(n-1)}4
=-\frac{3n}{4}+n\eta.
\]
Substituting into \eqref{eq:ordinaryX} and letting
$\eta\downarrow0$ gives
\begin{equation}\label{eq:ordinarybalanced}
t_2\ge\frac{n^2}{40}+n
+\frac14\sum_{L:\ s_L\ge4}s_L(s_L-4).
\end{equation}

Suppose next that $n/2<\ell\le7n/12$, and write $x\coloneq\ell/n$.
Now \eqref{eq:density} holds with $T=\ell/2$.
Lemma~\ref{lem:frame} therefore provides a normalization with
$S\preceq(\ell/2+\eta)I_4$, giving
\[
X\le n(\ell/2+\eta)-n-\frac{n(n-1)}4
=\frac12(x-\tfrac12)n^2-\frac{3n}{4}+n\eta.
\]
Since $\sum_{s_L\ge4}s_L(s_L-4)\ge\ell(\ell-4)$, using \eqref{eq:ordinaryX} and letting $\eta\downarrow0$ gives
\[
t_2\ge
\left[
\frac1{40}-\frac{11}{12}(x-\tfrac12)+\frac{x^2}{4}
\right]n^2+(1-x)n.
\]
The bracket is decreasing on $[1/2,7/12]$, and its value at
$7/12$ is $97/2880>1/40$. Hence $t_2>n^2/40$ in this case
as well. These cases prove the theorem.
\end{proof}

\subsection{Tightness of results}

The jointless cyclic Dowling configurations $Q'_4(t)$, $t\ge3$,
give an infinite family with $t_2=n^2/12$ and no plane
containing more than $n/2$ points; see
Section~\ref{sec:constructions}.
Hence, under the plane cap $n/2$, the optimal asymptotic
coefficient in Theorem~\ref{thm:introordinary} lies in $[1/40,1/12]$.

\paragraph{A construction over $\mathbb{R}$.}
An interesting construction over $\mathbb{R}$ can be obtained by modifying $Q'_4(t)$ by replacing the cyclic group $\mathbb{Z}/t\mathbb{Z}$ by a geometric progression.
De Zeeuw \cite[Section 4.1]{deZeeuw} remarked that a set of $n$ points in $\mathbb{R}^3$ evenly distributed on two skew lines determines $n^2/4$ ordinary lines with at most $n/2+1$ points in any plane, and he was unaware of any set determining fewer ordinary lines without being heavily concentrated on a single plane.
The following construction gives a smaller leading coefficient of $1/6$.

We work with projective coordinates. For an integer $q\ge2$, define
\[
P\coloneq\bigl\{[e_i-2^a e_j]:
1\le i<j\le4,\ a\in\Z,\ -q\le a\le q\bigr\}.
\]
Put $m\coloneq 2q+1$, so $n\coloneq|P|=6m$. We count the ordinary lines by
classifying their defining pairs according to their supports.

Points with the same support lie on one of the six coordinate
lines, each containing exactly $m\ge5$ points of $P$.
None of these lines is ordinary.
Points with disjoint supports determine an ordinary line:
a linear combination using both representatives has four nonzero
coordinates and therefore cannot represent another point of $P$.
There are three partitions of the four indices into two pairs,
giving $3m^2$ ordinary lines.

For overlapping supports, fix $i<j<k$.
The three points
\[
[e_i-2^a e_j],\qquad
[e_i-2^b e_k],\qquad
[e_j-2^c e_k]
\]
are collinear exactly when $b=a+c$, since the determinant of
their representatives is $2^b-2^{a+c}$.
Thus the number of three-point lines on these indices is
\[
T\coloneq\#\bigl\{(a,c)\in\{-q,\ldots,q\}^2:
-q\le a+c\le q\bigr\}
=\sum_{a=-q}^{q}(m-|a|)
=m^2-q(q+1).
\]
There are $3m^2$ pairs from different support classes,
and each three-point line accounts for three such pairs.
The number of ordinary lines on these indices is therefore
\[
3m^2-3T=3q(q+1).
\]
The four choices of three indices contribute $12q(q+1)$
ordinary lines.

These cases exhaust all defining pairs, so the total number
of ordinary lines is
\[
t_2(P)=3m^2+12q(q+1)
=24q(q+1)+3
=\frac{n^2}{6}-3.
\]

\paragraph{Optimality of individual steps in the argument.}
Determining the exact asymptotic constant in Theorem~\ref{thm:introordinary} is an interesting problem, and obtaining a tight bound is likely to require a new idea.
Lemma~\ref{lem:projectors} is optimal and the constants $11/4$ and $3/2$ in Lemma~\ref{lem:fourth} are best possible under the stated hypotheses.

We first consider Lemma~\ref{lem:fourth}.
The triple
$(e_1,e_1,e_2)$ has $X_I=1/2$ and $Y_I=11/8$, so an inequality $Y_I \leq \lambda X_I$ forces $\lambda \geq 11/4$, as in \eqref{eq:local3}. Triples of distinct
projective directions approach this example. 
An orthogonal pair has $X_I=-1/2$ and $Y_I=1/8$, so an
inequality $Y_I\le(11/4)X_I+c$ forces $c\ge3/2$,
as in \eqref{eq:local2}.

Next, we consider Lemma~\ref{lem:projectors}.
There exist five mutually unbiased orthonormal bases
$B_1,\ldots,B_5$ of $\C^4$; see the explicit construction in
\cite[Section~4]{BBRV}. Mutual unbiasedness means that
\[
|\langle u,v\rangle|^2=\frac14
\qquad
(u\in B_a,\ v\in B_b,\ a\ne b).
\]
Let $A\coloneq B_1\cup\cdots\cup B_5$. Its twenty vectors are pairwise
nonparallel, and
\[
\sum_{u\in A}uu^*=5I_4.
\]
Distinct vectors in the same basis are orthogonal, while vectors
from different bases contribute zero to $Y$. Therefore
\[
Y=5\cdot4\cdot3\left(\frac14\right)^2
=\frac{15}{4}
=\frac{3\cdot20^2}{80}-\frac{9\cdot20}{16},
\]
which precisely matches the bound of Lemma~\ref{lem:projectors}.
Repeating each vector $k$ times gives equality for $n=20k$.

\section{Acknowledgements}
The authors used GPT-6 Astra to assist with proof development and
writing. The authors take full responsibility for the mathematical
content and exposition of the paper. 

The first author is grateful to Xidian University for their hospitality and hosting him as a visitor in September-October of 2026.

\appendix
\section{Proof of the frame operator lemma}\label{app:frame}

We prove Lemma~\ref{lem:frame} using submodularity of linear rank,
the greedy basis algorithm, and Cauchy--Binet. This proof works
directly over $\C$ and includes boundary cases.

\subsection{Extending the lower weights}

For $I\subseteq[n]$, put
$r(I)\coloneq\dim\spanop\{v_i:i\in I\}$.
For $x \in \mathbb{R}^n$, put $x(I)\coloneq\sum_{i\in I}x_i$.
The density hypothesis implies $|I|/T\le r(I)$. Consequently the
compact set
\[
\mathcal K\coloneq\{x\in\R^n:x_i\ge1/T\text{ for all }i,
\ x(I)\le r(I)\text{ for all }I\subseteq[n]\}
\]
contains $(1/T,\ldots,1/T)$, and hence is nonempty. Compactness follows from $x_i\le r(\{i\})=1$.
Choose $c\in\mathcal K$ maximizing $c([n]) = \sum_i c_i$. We claim that
$c([n])=D$.

Call a set $I$ tight if $c(I)=r(I)$. If $I,J$ are tight,
submodularity gives
\begin{align*}
r(I)+r(J)&=c(I)+c(J)=c(I\cup J)+c(I\cap J)\\
&\le r(I\cup J)+r(I\cap J)\le r(I)+r(J).
\end{align*}
Thus their union is tight. Every index belongs to a tight set:
otherwise that coordinate could be increased a little, since there
are only finitely many constraints. Taking the union shows that
$[n]$ is tight. We have proved
\begin{equation}\label{eq:baseweights}
c_i\ge1/T,\qquad c(I)\le r(I),\qquad\sum_i c_i=D.
\end{equation}

\subsection{A log-determinant function bounded below}

For $z\in\R^n$, define
\[
M(z)\coloneq\sum_i e^{z_i}v_iv_i^*,\qquad
\Phi(z)\coloneq\log\det M(z)-c\cdot z.
\]
The matrix $M(z)$ is positive definite. If $\cB$ is the family of index sets $B \subseteq [n]$ such that $(v_i)_{i \in B}$ is a basis of $\mathbb{C}^D$, then we claim that
\begin{equation}\label{eq:cauchyb}
\det M(z)=\sum_{B\in\cB}a_Be^{\sum_{i\in B}z_i},
\qquad a_B\coloneq|\det(v_i:i\in B)|^2>0.
\end{equation}

Indeed, let $W\coloneq[e^{z_1/2}v_1 \ldots e^{z_n/2}v_n]$.
Then,
\[WW^* = \sum_{i=1}^n (e^{z_i/2}v_i)(e^{z_i/2}v_i)^* = M(z). \]
Consequently, by Cauchy--Binet,
\[\det M(z) = \sum_{B \in \binom{[n]}{D}} \det(W_B) \overline{\det(W_B)} = \sum_{B \in \binom{[n]}{D}} |\det(W_B)|^2. \]
For a fixed $B$, we can factor $e^{z_i/2}$ from column $i$ to obtain
\[ \det(W_B) = \left ( \prod_{i \in B} e^{z_i/2} \right) \det(V_B) = e^{1/2 \sum_{i \in B} z_i} \det(V_B), \]
where $V_B$ has columns $v_i, i \in B$. Hence,
\[|\det(W_B)|^2 = e^{\sum_{i \in B} z_i} |\det(V_B)|^2.\]
If $B$ is not a basis, $\det(V_B) = 0$. Keeping only the nonzero terms gives \eqref{eq:cauchyb}.

For every $z$, there exists a basis $B \in \cB$ that satisfies $\sum_{i\in B}z_i\ge c\cdot z$.
To verify this, relabel the indices, including the corresponding $v_i$ and $c_i$, so that $z_1\ge\cdots\ge z_n$, and take
the greedy basis in this order. For $I_j\coloneq\{1,\ldots,j\}$ it has
$|B\cap I_j|=r(I_j)$. Summation by parts gives
\begin{align*}
c\cdot z
&=Dz_n+\sum_{j<n}(z_j-z_{j+1})c(I_j)\\
&\le Dz_n+\sum_{j<n}(z_j-z_{j+1})r(I_j)
=\sum_{i\in B}z_i.
\end{align*}
Equation~\eqref{eq:cauchyb} therefore implies
\[
\Phi(z)\ge L\coloneq \min_{B\in\cB}\log a_B> -\infty.
\]

For $\rho>0$, let $z_\rho$ minimize
$\Phi(z)+\rho\|z\|^2/2$. The function is continuous and coercive,
so a minimizer exists. Its first-order condition and comparison
with $z=0$ give
\[
\nabla\Phi(z_\rho)=-\rho z_\rho,\qquad
\frac\rho2\|z_\rho\|^2\le\Phi(0)-L.
\]
Hence $\|\nabla\Phi(z_\rho)\|\le
\sqrt{2\rho(\Phi(0)-L)}\to0$ as $\rho\downarrow0$.
This argument does not require that $\Phi$ itself attain its infimum.

\subsection{From a small gradient to the frame bound}

Differentiating the log determinant, we have
\[
g_i\coloneq \partial_i\Phi(z)=e^{z_i}v_i^*M(z)^{-1}v_i-c_i.
\]
Choose $z$ so that $|g_i|<\delta$ for all $i$, with
$0<\delta<1/T$. Set $A\coloneq M(z)^{-1/2}$ and
$u_i\coloneq Av_i/\|Av_i\|$. Then
\[
I_D=AM(z)A^*
=\sum_i e^{z_i}\|Av_i\|^2u_iu_i^*
=\sum_i(c_i+g_i)u_iu_i^*
\succeq(1/T-\delta)\sum_i u_iu_i^*.
\]
Thus $S\preceq(1/T-\delta)^{-1}I_D$. Choose $\delta$ sufficiently
small that $(1/T-\delta)^{-1}\le T+\eta$. This proves the lemma.
\hfill$\square$

\phantomsection
\addcontentsline{toc}{section}{References}
\begingroup
\small
\setlength{\parskip}{0pt}
 \bibliography{frame_methods_references}
 \bibliographystyle{plain}

\endgroup
\end{document}